\documentclass{amsart}

\usepackage{amsmath, amsthm, amssymb}
\usepackage{amsfonts}
\usepackage{verbatim}
\usepackage{graphicx}
\usepackage{fancyhdr}
\usepackage{here}
\usepackage{psfrag}
\usepackage{epsfig}
\usepackage{pstricks,pst-node,pst-text,pst-3d}
\usepackage{epic}
\usepackage{eepic}
\numberwithin{equation}{section}

\newtheorem{theorem}{Theorem}
\newtheorem{lemma}{Lemma}
\newtheorem{corollary}{Corollary}

\theoremstyle{definition}
\newtheorem{rem}{Remark}

\newcommand{\cC}{\mathcal C}

\newcommand{\cR}{\mathcal R}

\newcommand{\cP}{\mathcal P}
\newcommand{\cU}{\mathcal U}
\newcommand{\cV}{\mathcal V}
\newcommand{\cW}{\mathcal W}
\newcommand{\cT}{\mathcal T}

\newcommand{\cD}{\mathcal D}

\newcommand{\cK}{\mathcal K}

\newcommand{\IR}{\mathbb R}

\newcommand{\vp}{\varphi}

\newcommand{\IP}{\mathbb P}

\DeclareMathOperator{\diam}{diam}

\title[cGFEM for hyperbolic integro-differential eqs.]
{Continuous Galerkin finite element methods for hyperbolic 
integro-differential equations}

\author[F.~Saedpanah]{Fardin Saedpanah}

\address{
Department of Mathematics, 
University of Kurdistan, P. O. Box 416, 
Sanandaj, Iran
}

\email{f.saedpanah@uok.ac.ir\\
           f\_saedpanah@yahoo.com}

\keywords{integro-differential equation, linear semigroup theory, 
continuous Galerkin finite element method, 
convolution kernel, stability, a priori estimate.}

\subjclass{65M60, 45K05}

\begin{document}

\begin{abstract}
A hyperbolic integro-differential equation is considered, as a model problem, 
where the convolution kernel is assumed to be either smooth or no worse 
than weakly singular. 
Well-posedness of the problem is studied in the 
context of semigroup of linear operators, and regularity of any order 
is proved for smooth kernels. 
Energy method is used to prove optimal order a priori error estimates 
for the finite element spatial semidiscrete problem. 
A continuous space-time finite element method of order one 
is formulated for the problem. Stability of the discrete dual problem 
is proved, that is used to obtain optimal order a priori estimates 
via duality arguments. The theory is illustrated by an example. 
\end{abstract}

\date{March 8, 2013}

\maketitle

\section{Introduction}
We consider, for any fixed $T>0$,  a hyperbolic type integro-differential 
equation of the form
\begin{equation}  \label{Problem}
  \ddot u+A u-\int_0^t \cK(t-s)Au(s)\,ds=f, \quad t \in (0,T), 
  \quad {\rm with} \ u(0)=u^0,\ \dot{u}(0)=u^1,
\end{equation}
(we use `$ \cdot $' to denote `$ \frac{\partial}{\partial t} $') 
where $A$ is a self-adjoint, positive definite, uniformly elliptic 
second order operator on a Hilbert space. 
The kernel $\cK$ is considered to be either smooth (exponential), 
or no worse than weakly singular, and in both cases with the properties that
\begin{align}  \label{KernelProperty}
  \cK\geq 0,\quad \dot\cK(t) \leq 0, \quad \|\cK\|_{L_1(\IR^+)}=\kappa<1.
\end{align}

This kind of problems arise e.g., in the thoery of linear and fractional order 
viscoelasticity. 
For examples and applications of this type of problems see, e.g., 
\cite{RiviereShawWhiteman2007}, \cite{StigFardin2010}, 
and references therein.

For our analysis, we define a function $\xi$ by 
\begin{align}   \label{xi}
  \xi(t)=\kappa-\int_0^t\!\cK(s)\,ds
    =\int_t^\infty\!\cK(s)\,ds,
\end{align}
and, having \eqref{KernelProperty}, it is easy to see that
\begin{align}   \label{xiproperty}
  \begin{split}
    D_t\xi(t)=-\cK(t)<0,\quad
    \xi(0)=\kappa,\quad \lim_{t\to\infty}\xi(t)=0,\quad
    0<\xi(t)\le\kappa.
  \end{split}
\end{align}
Hence, $\xi$ is a completely monotone function, since
\begin{align*}
  (-1)^jD_t^j\xi(t)\ge 0,\quad t\in (0,\infty),\,j=0,1,2,
\end{align*}
and consequently $\xi\in L_{1,loc}[0,\infty)$ is a positive type kernel, 
that is, for any $T\ge 0$ and $\phi\in \cC([0,T])$,
\begin{align}   \label{positivetype}
  \int_0^T\!\int_0^t\!\xi(t-s)\phi(t)\phi(s)\,ds\,dt\ge 0 .
\end{align}

From the extensive literature on theoritical and numerical analysis 
for partial differential equations with memory, we mention 
\cite{RiviereShawWhiteman2007}, \cite{StigFardin2010}, 
\cite{AdolfssonEnelundLarsson2008}, 
\cite{McLeanThomee2010}, \cite{FardinarXiv:1205.0159}, 
and their references. 

The fractional order kernels, such as Mittag-Leffler type kernels in 
fractional viscoelasticity, interpolate between smooth (exponential) 
kernels and weakly singular kernels, that are singular at origin but 
integrable on finite time intervals $(0,T)$, for any $T \geq 0$, see 
\cite{FardinarXiv:1203.4001} and references therein.  
This is the reason for considering problem \eqref{Problem} with 
convolution kernels satisfying \eqref{KernelProperty}. 

In \cite{StigFardin2010} well-posedness of a problem, similar to 
\eqref{Problem} with a Mittag-Leffler type kernel, 
was studied in the framework of the linear semigroup 
theory. 
Here we first extend the theory to prove 
higher regularity of the solution for more smooth kernels, 
such that  a priori error estimates are fulfilled. 
We prove $L_{\infty}(L_2)$ optimal 
order a priori error estimate, by energy methods, for finite element 
spatial semidiscrete approximate solution. 
This provides an alternative proof to what 
we presented in \cite{StigFardin2010}, and is straightforward.
The  continuous space-time finite element method of order one, cG(1)cG(1), 
is used to formulate the fully dicrete problem.  
A similar method has been applied to the wave equation in 
\cite{BangerthGeigerRannacher2010}, where adaptive methods based on 
dual weighted residual (DWR) method has been studied. 
An energy identity is proved for the discrete dual 
problem, using the positive type auxiliary function $\xi$. 
This is then used to prove $L_{\infty}(L_2)$ and $L_{\infty}(H^1)$ 
optimal order a priori error estimates by duality.  
This and  \cite{FardinarXiv:1205.0159}, 
where a posteriori error analysis of this method has been studied 
via duality, complete the error analysis of this method for 
model problems similar to \eqref{Problem}.  

The present work also extend previous works, e.g., 
\cite{AdolfssonEnelundLarsson2008}, \cite{AdolfssonEnelundLarsson2004}, 
\cite{ShawWhiteman2004}, on quasi-static fractional order viscoelsticity 
$(\ddot{u}\approx 0)$ to the dynamic case. 
Spatial finite element approximation of integro-differential equations  
similar to \eqref{Problem} have been studied in 
\cite{AdolfssonEnelundLarssonRacheva2006} and 
\cite{LinThomeeWahlbin1991}, however, for optimal order 
$L_\infty(L_2)$ a priori error estimate for the solution $u$, 
they require  one extra time derivative regularity of the solution.
A dynamic model for viscoelasticity based on internal variables 
is studied in \cite{RiviereShawWhiteman2007}. 
The memory term generates a growing amount of 
data that has to be stored and used in each time step. This can be dealt with by 
introducing ``sparse quadrature'' in the convolution term 
\cite{SloanThomee1986}. 
For a different approach based on ``convolution quadrature'', see \cite{SchadleLopezLubich2008}. 
However, we should note that this is not an issue for exponentially decaying 
memory kernels, in linear viscoelasticity, that are represented as a Prony series.  
In this case recurrence relationships can be derived which means 
recurrence formula are used for history updating, 
see \cite{ShawWhiteman2004}  
and  \cite{KaramanouShawWarbyWhiteman2005} for more details. 
In practice, the global regularity needed for a priori error analysis is not 
present, e.g., due to the mixed boundary conditions, that calls for 
adaptive methods based on a posteriori error analysis. 
We plan to address these issues in future work.

In the sequel, in $\S2$, well-posedness of the problem is proved and 
high regularity of the solution of the problem with smooth kernels is verified. 
In $\S3$, the spatial finite element discretization is studied and, 
using energy method, optimal order a priori error estimates are proved. 
The continuous space-time finite element method of order one is applied to 
the problem in $\S4$, and stability estimates for the discrete dual 
problem are obtained. These are then used to prove optimal order 
 a priori error estimates in $\S5$  by duality. Finally, in $\S6$, 
 we illustrate the theory by a simple example. 

\section{Well-posedness and regularity}      
We use the semigroup theory of linear operators to show that 
there is a unique solution of \eqref{Problem},  and we prove that 
under appropriate assumptions on the data we get higher regularity 
of the solution.  
In $\S2.1$ we quote the main framework from \cite{StigFardin2010}, 
to prove existence and uniqueness, to be complete. 
Here we restrict to pure homogeneous  Dirichlet boundary condition, 
though the presented framework applies also to mixed homogeneous  
Dirichlet-Neumann boundary conditions. But it does not admit 
mixed homogeneuos Dirichlet nonhomogeneous Neumann boundary 
conditions, and this case has been  studied in \cite{FardinarXiv:1203.4001} 
for a more general problem, by means of Galerkin approximation method. 
Then in $\S2.2$ we extend the semigroup framework to prove regularity 
of any  order for models with smooth kernels. To this end, we specialize 
to the  homogeneous  Dirichlet boundary condition.  

\subsection{Existence and uniqueness}      
We let $\Omega \subset\mathbb{R}^d$, be a bounded convex 
domain with smooth boundary $\partial\Omega$. 
In order to describe the spatial regularity of functions, 
we recall the usual Sobolev spaces $H^s=H^s(\Omega)^d$ 
with the corresponding norms and inner products, 
and we denote  $H=H^0=L_2(\Omega)^d,\ V=H_0^1(\Omega)^d$. 
We equip $V$ with the energy inner product 
$a(u,v)=(Au,v)$ and norm $\|v\|_V^2=a(v,v)$. 
We recall that $A$ is a selfadjoint, positive definite, 
unbounded linear operator, with 
$\cD(A)=H^2\cap V$, and we use the norms 
$\|v\|_s=\|A^{s/2}v\|$. 
We note that with  mixed homogeneous  Dirichlet-Neumann 
boundary conditions, we have  
\begin{equation*}
  V=\{v \in H^1 : v=0 \text{\ on Dirichlet boundary}\}.
\end{equation*}

We extend $u$ by $u(t)=h(t)$ for $t<0$ with $h$ to be chosen.
By adding $-\int_{-\infty}^0 K(t-s)Ah(s)\,ds$ to both sides of
\eqref{Problem}, changing the variables in the convolution terms and defining
$w(t,s)=u(t)-u(t-s)$, we get
\begin{equation} \label{ReformedProblem}
   \ddot u(t)+(1-\kappa) Au(t)
   +\int_0^\infty\! \cK(s) A w(t,s)\,ds
   = f(t)-\int_t^{\infty}\! \cK (s)Ah(t-s)\,ds,
\end{equation}
where, we recall that $\|\cK\|_{L_1(\IR^+)}=\kappa<1$. 
For latter use, we note that equation \eqref{Problem} can be retained from 
\eqref{ReformedProblem} by backward calculations.

For a given integer number $r \geq 0$, we use the Taylor expansion 
of order $r$ of the solution $u$ at  $t=0$ to define the extension $u(t)=h_r(t)$ for $t<0$. 
That is, we set 
\begin{equation}  \label{h-r}
  u(t)=h_r(t)=\sum_{n=0}^r \frac{t^n}{n!} u^n(0),\quad t<0,
\end{equation}
where we use the notation $u^n(t)=u^n(t,\cdot)=\frac{\partial^n}{\partial t^n}u(t,\cdot)$, 
with  $u^0(t)=u(t)$. 

Now we reformulate the model problem \eqref{Problem} 
to an abstract Cauchy problem. 
First, we choose $r=0$ in \eqref{h-r}, that is $h_0(t)=u^0$, and 
for the initial data we assume that $u^0 \in \cD(A)$ and $u^1 \in V$. 
Therefore, from \eqref{ReformedProblem}, we have
\begin{equation}\label{eq13}
    \ddot u(t)+(1-\kappa)Au(t)
      +\int_0^\infty\! \cK (s)Aw(t,s)\,ds
      =f(t) - Au^0\int_t^\infty\! \cK (s)\ ds,
\end{equation}
where,
\begin{equation*}
  w(t,s)=
  \left\{
  \begin{aligned}
    &u(t)-u(t-s),&&s\in[0,t],\\
    &u(t)-u^0,&&s\in[t,\infty).
  \end{aligned}
  \right.
\end{equation*}
Then we write \eqref{eq13}, together with the initial conditions, 
as an abstract Cauchy problem and prove well-posedness. 

We set $v=\dot u$ and define the Hilbert spaces
\begin{equation*}
  \begin{split}
    W&=L_{2, \cK }\big(\IR^+;V\big)
        =\Big\{w: \lVert w \rVert_W^2
          =\int_0^\infty\! \cK (s)\lVert w(s)\rVert_V^2\,ds<\infty\Big\},\\
    Z&=V\times H\times W
      =\Big\{z=(u,v,w):
       \lVert z \rVert_Z^2=(1-\kappa)\lVert u\rVert_V^2+
          \lVert v\rVert^2+\lVert w \rVert_W^2<\infty\Big\}.
  \end{split}
\end{equation*}
We also define the linear operator $\mathcal{A}$ on $Z$
such that, for
$z=(u,v,w),$
\begin{equation*}
  \mathcal{A}z=\bigg(v\,,\,
     -A\Big((1-\kappa) u+\int_0^\infty \! \cK(s)w(s)\,ds\Big)
                \,,\,v-Dw\bigg),
\end{equation*}
with domain of definition
\begin{equation*}
  \mathcal{D}(\mathcal{A})
  =\Big\{(u,v,w)\in Z: 
  v\in V,\,
    (1-\kappa) u+\int_0^\infty\! \cK(s)w(s)\,ds\in\mathcal{D}(A),
     w\in \mathcal{D}(D)\Big\}.
\end{equation*}
Here $Dw=\frac{d}{ds}w$ with 
$  \mathcal{D}(D)=\{w\in W : Dw\in W\,\,{\rm and}\,\,w(0)=0\}.$

Therefore, a solution of \eqref{Problem} satisfies the 
system of delay differential equations, for $t \in (0,T)$,
\begin{equation*}
  \begin{split}
    \dot u(t)&=v,\\
    \dot v(t)&=-A\Big((1-\kappa) u(t)
      +\int_0^\infty \! \cK(s)w(t,s)\,ds\Big)
      +f(t) - Au^0\int_t^\infty\! \cK (s)\ ds,\\
    \dot w(t,s)&=v(t)-D w(t,s),\quad s \in (0,\infty).
  \end{split}
\end{equation*}
This can be writen as the abstract Cauchy problem
\begin{equation}\label{abstractCauchyproblem}
  \begin{split}
     &\dot z(t)=\mathcal{A}z(t)+F(t),\quad t \in (0,T),\\
     &z(0)=z^0, 
  \end{split}
\end{equation}
where $F(t)=(0,f(t) - Au^0\int_t^\infty\! \cK (s)\ ds,0)$ and
$z^0=\big(u^0, u^1,0\big)$, since
\begin{equation*}
   w(0,s)=u(0)-u(-s)=u(0)-h(-s)=u^0-u^0=0.
\end{equation*}
We note that $ w(t,0)= u(t)- u(t)=0$, 
so that $ w(t,\cdot)\in \mathcal{D}(D)$.

We quote from \cite[Theorem 2.2]{StigFardin2010}, that 
$\mathcal{A}$ generates a $C_0$-semigroup of cotractions on $Z$. 
\begin{corollary} 
The linear operator $\mathcal{A}$ is an infinitesimal generator of a 
$C_0$-semigroup $e^{t\mathcal{A}}$ of contractions on the Hilbert space $Z$. 
\end{corollary}

Now, we look for a strong solution of the initial value 
problem \eqref{abstractCauchyproblem}, that is, 
a function $ z$ which is differentiable a.e.\ on $[0,T]$ with
$\dot{ z}\in L_1((0,T);Z)$, if $ z(0)= z^0$, $ z(t)\in \mathcal{D}(\mathcal{A})$,
and 
$\dot{ z}(t)=\mathcal{A} z(t)+F(t)$ a.e.\ on $[0,T]$.

Recalling the assumptions $u^0 \in \cD(A)$ and $u^1 \in V$, 
we know that if $z=(u, v, w)$ be a strong solution of the abstract 
Cauchy problem \eqref{abstractCauchyproblem} with 
$ z^0=\big( u^0,u^1,0\big)$,  
then $ u$ is a solution of \eqref{Problem} 
by \cite[Lemma 2.1]{StigFardin2010}. 
Hence, to prove that there is a unique solution for \eqref{Problem}, 
we need to prove that there is a unique strong solution for 
\eqref{abstractCauchyproblem}. This has been proved in 
\cite[Theorem 2.2]{StigFardin2010}, 
if  $f:[0,T]\to H$ is Lipschitz continuous, using the fact that 
the linear operator $\mathcal{A}$ generates a $C_0$-semigroup 
of contractions on $Z$. 
Moreover, for some 
$C=C(\kappa,T)$, 
we have the regularity estimate, for $t \in [0,T]$,
\begin{equation}  \label{RegularityEstimate-1}
  \lVert  u(t)\rVert_V+\lVert\dot{ u}(t)\rVert 
  \le C\Big(\lVert A u^0\rVert+\lVert u^1\rVert
       +\int_0^t\!\lVert f\rVert \,ds\Big).
\end{equation}

\subsection{High order regularity}       
In order to prove higher regularity of order $r+1$ ($\ r \geq 1$), 
we assume that the bounded domain $\Omega$ is convex, 
and we specialize to the homogeneous Dirichlet 
boundary condition. Hence, the elliptic regularity estimate 
holds, that is
\begin{equation}   \label{ellipticreg}
  \| u\|_2 \le C\|Au\|,\quad  u\in \cD(A)=H^2 \cap V.
\end{equation}

We note that the case $r=0$ is the choice for \eqref{RegularityEstimate-1}. 
We substitute $h_r(t)$ from \eqref{h-r}, with $r\geq 1$, 
in \eqref{ReformedProblem}. 
Then, differentiating $\frac{\partial^r}{\partial t^r}$ and using the notation 
$u^r(t)=\frac{\partial^r}{\partial t^r}u(t)$, we have 
\begin{equation} \label{r-thDerivativeReformed}
  \begin{split}
    \ddot u^r(t)&+(1-\kappa) Au^r(t)
      +\int_0^\infty\!\!  \cK (s)Aw^r(t,s)\,ds\\
    &\qquad =f^r(t)-A\frac{\partial^r}{\partial t^r} 
                     \int_t^\infty \! \cK(s)\sum_{n=0}^{r}\frac{(t-s)^n}{n!}u^n(0)\ ds\\
    &\qquad =f^r(t)+A\sum_{n=0}^{r-1}u^n(0) \cK ^{r-n-1}(t)
                     -Au^r(0) \xi(t) \\
    &\qquad =:\tilde f^r(t),
  \end{split}
\end{equation}
with the initial data $u^r(0), u^{r+1}(0)$. 

Recalling the initial data $u(0)=u^0$ and $u^1(0)=u^1$, from \eqref{Problem}, 
we have $u^2(0)=f(0)-Au^0 $. 
To obtain $u^m(0),\ m\geq 3$, we differentiate 
$\frac{\partial^{m-2}}{\partial t^{m-2}}$ of equation \eqref{Problem}, 
and we have
\begin{equation}  \label{r-thDerivative}
  \begin{aligned}
     \ddot u^{m-2}(t)+A u^{m-2}(t)&-\int_0^t \cK^{m-2}(t-s)Au(s)\,ds\\
     &=f^{m-2}(t) +A\sum_{n=0}^{m-3}u^n(t)\cK^{m-n-3}(0), 
     \quad t\in(0,T),
  \end{aligned}
\end{equation} 
that, with $t=0$, implies the initial condition
\begin{equation}  \label{Initial-r2}
  \begin{aligned}
    &u^{m}(0)=f^{m-2}(0)-Au^{m-2}(0)
    +\sum_{n=0}^{m-3}Au^{n}(0)\cK^{m-n-3}(0),
       &&m \geq 3.
  \end{aligned}
\end{equation} 

Throughout, obviously any sum $\sum_{n=i}^j$ is supposed to be suppressed 
from the formulas, when $i>j$. 
\begin{rem} \label{SobolevInequality}
We note that, if we assume $\cK \in W_1^{m-2}(0,T)$, 
then $\cK \in \mathcal{C}^{m-3}[0,T]$ by Sobolev inequality, 
and therefore $u^{m}(0)$ is well-defined.  
\end{rem}

\begin{rem} \label{InitialInequality}
One can show, by induction and the fact that by \eqref{ellipticreg}
\begin{equation} \label{InitialInequality-0}
  \|v\| \leq \|v\|_{H^2} \leq C\|Av\|, \quad \text{for}\ v \in \mathcal{D}(A), 
\end{equation}
we have $(m=0,1,\ k=1,2,\cdots)$,
\begin{equation} \label{InitialInequality-even}
  \begin{split}
    | A^m u^{2k}(0)| &\leq C\Big( |A^{m+k}u^0| + |A^{m+k-1}u^1| \\
      &\quad 
      + \sum_{j=0}^{k-1}|A^{m+j}f^{2k-2j-2}(0)| 
       + \sum_{j=1}^{k-2}|A^{m+j}f^{2k-2j-3}(0)|\Big),
  \end{split}
\end{equation}
\begin{equation} \label{InitialInequality-odd} 
  \begin{split}
    | A^m u^{2k+1}(0)| &\leq C\Big( |A^{m+k}u^0| + |A^{m+k}u^1| \\
      &\quad 
      + \sum_{j=1}^{k-1}|A^{m+j}f^{2k-2j-2}(0)|   
       + \sum_{j=0}^{k-1}|A^{m+j}f^{2k-2j-1}(0)|\Big).
  \end{split}
\end{equation}

\end{rem}

Now we note that, in \eqref{r-thDerivativeReformed}, we have
\begin{equation*}
   w^r(t,s)=
  \left\{
  \begin{aligned}
    & u^r(t)- u^r(t-s),&&s\in[0,t],\\
    & u^r(t)- u^r(0),&&s\in[t,\infty),
  \end{aligned}
  \right.
\end{equation*}
so that $w^r(t,0)=0$. Therefore, considering continuty of $w^r$, 
we have $w^r \in \mathcal{D}(D)$. 

Then, in the same way as in the previous section, with $v^r= \dot u^r$, 
we can reformulate \eqref{r-thDerivativeReformed}, with $z^r=(u^r,v^r,w^r)$, as 
the abstract Cauchy problem
\begin{equation} \label{abstractCauchyproblem-r}
  \begin{split}
     &\dot z^r(t)=\mathcal{A}\, z^r(t)+F^r(t),\quad 0<t<T,\\
     &z^r(0)= z^{r,0},
  \end{split}
\end{equation}
where 
$F^r(t)=(0,\tilde f^r(t),0)$ and
$z^{r,0}=(u^r(0),u^{r+1}(0),0)$, 
since 
$w^r(0,s)=u^r(0)- u^r(0)=0$.

In particular, for $r=1$, we have 
\begin{equation*}
  F^1(t)=\Big(0,f^1(t)+Au^0 \cK(t)-Au^1 \xi(t),0\Big),
\end{equation*}  
with initial data 
$z^{1,0}=(u^1,u^2(0),0)=(u^1,f(0)-Au^0,0)$.

Now, we need to show that from a strong solution of the abstract 
Cauchy problem \eqref{abstractCauchyproblem-r}, for $r \geq 1$, 
we get a solution of the main problem \eqref{Problem}. 
Therfore we should prove that the abstract 
Cauchy problem \eqref{abstractCauchyproblem-r} has a unique 
strong solution, under certain conditions on the data. 
The proof is by induction, and therefore we recall some facts from 
\cite{StigFardin2010}, for $r=1$. 
\begin{lemma} \label{lemma-r:1}
Let $ z^1=( u^1, v^1, w^1)$ be a strong solution of the abstract 
Cauchy problem \eqref{abstractCauchyproblem-r} with 
$ z^{1,0}=\big( u^1,u^2(0),0\big)$. 
Then $ u(t)= u^0+\int_0^t u^1(s)\,ds$ is a solution of \eqref{Problem}.
\end{lemma}

\begin{theorem} \label{theorem-r:1}
There is a unique solution $u=u(t)$ of
\eqref{Problem} if 
$\cK \in W_1^1(\IR^+)$ with $\|\cK\|_{W_1^1(\IR^+)}<1$, 
$f(0)-Au^0\in V,\ u^1\in \mathcal{D}(A)$, and
$\dot f:[0,T]\to H$ is Lipschitz continuous. 
Moreover,
for some $C=C(\kappa,T)$, we have the regularity estimate, 
for $t \in [0,T]$,
\begin{equation} \label{regularityestimate2}
  \lVert \dot u(t)\rVert_V+\lVert\ddot u(t)\rVert 
  \le C\Big(\lVert Au^0\rVert+
     \lVert Au^1 \rVert 
     + \lVert f(0) \rVert
      +\int_0^t\!\lVert \dot f\rVert \,ds\Big).
\end{equation}
\end{theorem}
\begin{proof}
There exists a unique strong solution $z^1=(u^1,v^1,w^1)$ 
for \eqref{abstractCauchyproblem-r}, with $r=1$, by 
\cite[Theorem 2.4]{StigFardin2010}. 
Hence, the proof is complete by Lemma \ref{lemma-r:1}.
\end{proof}
\begin{lemma} \label{lemma-r}
Let $z^r=(u^r,v^r,w^r)$, for $r\geq 1$, be a strong solution of 
the abstract Cauchy problem  \eqref{abstractCauchyproblem-r} 
with $z^{r,0}=(u^r(0),u^{r+1}(0),0)$. 
Then $ u(t)= u^0+\int_0^t u^1(s)\,ds$ is a solution of
\eqref{Problem}.
\end{lemma}
\begin{proof}
The proof is by induction. The case $r=1$ follows from 
Theorem \ref{theorem-r:1}. 

Now, we assume that the lemma valids for some $r \geq 2$, 
and we prove that it holds also for $r+1$. 
To this end, we show that if $z^{r+1}=(u^{r+1},v^{r+1},w^{r+1})$ 
be a strong solution of \eqref{abstractCauchyproblem-r} (for $r+1$) 
with $z^{r+1,0}=(u^{r+1}(0),u^{r+2}(0),0)$, 
then $z^r=(u^r,v^r,w^r)$ is a strong solution of \eqref{abstractCauchyproblem-r} 
with $z^{r,0}=(u^r(0),u^{r+1}(0),0)$, that completes the proof by induction assumption. 

Since $\dot z^{r+1}(t)=\mathcal{A}\, z^{r+1}(t)+F^{r+1}(t)$ a.e.\ on $[0,T]$, 
we have, for $t \in (0,T)$,
\begin{equation*}
  \begin{split}
    &\dot u^{r+1}(t)=v^{r+1}(t),\\
    &\dot v^{r+1}(t)=-A\Big((1-\kappa) u^{r+1}(t)
     +\int_0^\infty\! \cK(s)w^{r+1}(t,s)\,ds\Big)
     +\tilde f^{r+1}(t),\\
    &\dot w^{r+1}(t,s)
     =v^{r+1}(t)-Dw^{r+1}(t,s),
      \  s\in(0,\infty).
  \end{split}
\end{equation*}
The first and the third equation implies that $w^{r+1}$ satisfies 
the first order partial differential equation 
\begin{equation*}
  w^{r+1}_t+w^{r+1}_s=u^{r+1}_t.
\end{equation*}
This, with $w^{r+1}(t,0)=0,\,w^{r+1}(0,s)=0$, 
has the unique solution 
$w^{r+1}(t,s)=u^{r+1}(t)-u^{r+1}(t-s)$, that implies,
by integration with respect to $t$, 
\begin{equation*}
  w^r(t,s)=u^r(t)-u^r(t-s)=\int_0^t\! w^{r+1}(\tau,s)\,d\tau.
\end{equation*}

From the first and the second equations we obtain equation  
\eqref{r-thDerivativeReformed} with $r+1$, that is obtained from equation 
\eqref{ReformedProblem} by differentiating $\frac{\partial^r}{\partial t^r}$. 
We recall that equations \eqref{Problem} and 
\eqref{ReformedProblem} are equivalent, that implies 
equivalence of equations \eqref{r-thDerivativeReformed} 
and \eqref{r-thDerivative}. 
Therefore $u$ also satisfies \eqref{r-thDerivative} with $r+1$. 
Then, integrating with respect to $t$, we have, for $t \in (0,T)$,
\begin{equation}  \label{Thm2-eq1}
  \begin{aligned}
     \ddot u^r(t)
     -\ddot u^r&(0)+A u^r(t)-A u^r(0)
       -\int_0^t \int_0^\tau \! \cK^{r+1}(\tau-s)Au(s)\ ds\ d\tau\\
     &=f^r(t) -f^r(0)
         +A\sum_{n=0}^{r} \int_0^t \! u^n(\tau)\ d\tau \cK^{r-n}(0).
  \end{aligned}
\end{equation} 
Now, we need to show that \eqref{Thm2-eq1} implies \eqref{r-thDerivative}. 
We note that
\begin{equation*} 
  \begin{aligned}
     \int_0^t \int_0^\tau \! \cK^{r+1}(\tau-s)Au(s)\ ds\ d\tau
     &=\int_0^t \int_s^t \! \cK^{r+1}(\tau-s)Au(s)\ d\tau \ ds\\
     &=\int_0^t \! \cK^{r}(t-s)Au(s)\ ds-A \cK^{r}(0)\int_0^t \! u(s)\ ds,
  \end{aligned}
\end{equation*} 
and 
\begin{equation*} 
  \begin{aligned}
     A\sum_{n=0}^{r} &\int_0^t \! u^n(\tau)\ d\tau \cK^{r-n}(0)\\
     &=A\bigg(
         \int_0^t\!u(\tau)\ d\tau \cK^r(0)
         +\sum_{n=1}^r \big(u^{n-1}(t)-u^{n-1}(0)\big) \cK^{r-n}(0)\bigg)\\
     &=A \cK^r(0)\int_0^t\!u(\tau)\ d\tau 
         +A\sum_{n=0}^{r-1}u^{n}(t) \cK^{r-n-1}(0)
         -A\sum_{n=0}^{r-1}u^{n}(0) \cK^{r-n-1}(0).
  \end{aligned}
\end{equation*} 
Using these and \eqref{Initial-r2} in \eqref{Thm2-eq1} we 
conclude \eqref{r-thDerivative}, that is equivalent to 
\eqref{r-thDerivativeReformed}. 
This means that, $z^r=(u^r,v^r,w^r)$ is a strong solution of 
\eqref{abstractCauchyproblem-r} with $z^{r,0}=(u^r(0),u^{r+1}(0),0)$. 
Hence, by induction assumption, $ u(t)= u^0+\int_0^t u^1(s)\,ds$ 
is a solution of \eqref{Problem}, 
and this completes the proof.
\end{proof}

In the next theorem we find the circumstances under which 
there is a unique strong solution of the abstract Cauchy problem 
\eqref{abstractCauchyproblem-r}, that by Lemma \ref{lemma-r} 
implies existence of a  unique solution of \eqref{Problem} 
with higher regularity.  
We also obtain regularity estimates, which are extensions of \eqref{RegularityEstimate-1} and \eqref{regularityestimate2}.

We note that, recalling Remark \ref{SobolevInequality} and having 
the assumptions from the next theorem, the 
 calculations in the proof of Lemma \ref{lemma-r} make sense. 
\begin{theorem} \label{theorem-r}
For a given integer number $r \geq 1$, 
let $f^r=\frac{\partial^r}{\partial t^r}f:[0,T]\to H$ be Lipschitz
 continuous and  $\cK \in W^r_1(\IR^+)$ with 
 $\|\cK\|_{W^r_1(\IR^+)}<1$. 
We also, recalling $\mathcal{D}(A)= H^{2}\cap V$, assume the 
following compatibility conditions:

\noindent
for $r=2k \,(k=1,2,\cdots)$, 
\begin{equation}\label{r-even}
  \begin{split}
    &A^ku^0\in \mathcal{D}(A),\quad A^k u^1\in V,\\
    &f^{r-2j}(0)\in H^{2j} \cap V, \quad 
       j=1,\cdots,\,k,\ k\geq 1,\\
    &f^{r-2j+1}(0)\in H^{2(j-1)} \cap V, \quad
       j=1,\cdots,\,k,\ k\geq 1,
  \end{split}
\end{equation}
and for $r=2k+1 \,(k=0,1,2,\cdots)$, 
\begin{equation}\label{r-odd}
  \begin{split}
    &A^k u^1\in \mathcal{D}(A),\quad A^{k+1} u^0\in V,\\
    &f^{r-2j}(0)\in H^{2j} \cap V, \quad j=1,\cdots,k,\ k\geq 1,\\
    &f^{r-2j+1}(0)\in H^{2(j-1)} \cap V, \quad j=1,\cdots,k+1,\ k\geq 0.
  \end{split}
\end{equation}
Then there is a unique solution of \eqref{Problem}.

Moreover, for some $C=C(\kappa,\|\cK\|_{W^r_1(\IR^+)},T)$:
 
\noindent
for $r=2k \,(k=1,2,\cdots)$, we have the regularity estimate
\begin{equation} \label{evengeneralregularityestimate}
  \begin{split}
    \lVert u^r(t) \rVert_V&+\lVert u^{r+1}(t)\rVert\\
    & \le C\bigg(\lVert A^{k+1}u^0\rVert +\lVert A^k u^1 \rVert \\
    &\quad  +\sum_{j=1}^k \lVert A^j f^{r-2j}(0)\rVert
               +\sum_{j=1}^k \lVert A^{j-1} f^{r-2j+1}(0)\rVert 
    +\int_0^t\!\lVert f^r(s)\rVert \,ds\bigg),
  \end{split}
\end{equation}
and,  for $r=2k+1 \,(k=0,1,2,\cdots)$, we have the estimate
\begin{equation} \label{oddgeneralregularityestimate}
  \begin{split}
    \lVert u^r(t) \rVert_V&+\lVert u^{r+1}(t)\rVert\\
    &\!\!\!\!\!\le C\bigg(\lVert A^{k+1}u^0\rVert+\lVert A^{k+1}u^1 \rVert \\
    &+\sum_{j=1}^k \lVert A^j f^{r-2j}(0)\rVert
       +\sum_{j=1}^{k+1} \lVert A^{j-1} f^{r-2j+1}(0)\rVert 
    +\int_0^t\!\lVert f^r(s)\rVert \,ds\bigg).
  \end{split}
\end{equation}
\end{theorem}
\begin{proof}
1. 
The case $r=1$ follows from Theorem \ref{theorem-r:1}. 
Then, for a given $r \geq 2$, we show that 
\begin{itemize}
\item[(i)] 
$z^r(0)=z^{r,0}=(u^r(0),u^{r+1}(0),0) \in \mathcal{D}(\mathcal{A})$,
\item[(ii)]
$F^r$ is differentiable almost everywhere on $[0,T]$ and 
$\dot F \in L_1([0,T];Z)$. 
\end{itemize}
These imply existence of a unique strong solution of the abstract 
Cauchy problem \eqref{abstractCauchyproblem-r}, 
by \cite[Corollary 4.2.10]{Pazy}, that yields existence 
of a unique solusion of \eqref{Problem}, by Lemma \ref{lemma-r}.

2. First we note that (i) holds, if   
$u^r(0) \in \mathcal{D}(A)$ and $u^{r+1}(0) \in V$, 
by the definition of $\mathcal{D}(\mathcal{A})$.  
This can be verified by applying the compatibility conditions 
\eqref{r-even}--\eqref{r-odd} in 
\eqref{InitialInequality-even}--\eqref{InitialInequality-odd}, using 
\eqref{InitialInequality-0}. 

3. 
Now we prove (ii). By assumption $f^r:[0,T] \to H$ is Lipschitz continuous. 
Therefore, by a classical result from functional analysis, 
$f^r$ is differentiable almost everywhere on $[0,T]$ and 
$\dot f^r \in L_1([0,T];H)$, since $H$ is a Hilbert space. 
Then, recalling the assumption $\cK \in W_1^r(\IR^+)$  
and the fact that 
\begin{equation*}
  \dot{\tilde{f}}^r(t)
  =\dot f^r(t)+A \sum_{n=0}^r u^n(0) \cK^{r-n}(t),
\end{equation*}
we conclude that $F^r(t)=(0,\tilde f^r(t),0)$ is differentiable almost everywhere 
on $[0,T]$ and $\dot F^r \in L_1([0,T];Z)$, that completes the proof of (ii).  

4. 
Hence, since $\mathcal{A}$ generates a $C_0$-semigroup of 
contractions on $Z$ by Corollary 1, we conclude, 
by \cite[Corollary 4.2.10]{Pazy}, 
that there exists a unique strong solution $z^r=(u^r,v^r,w^r)$ for the 
abstract Cauchy problem \eqref{abstractCauchyproblem-r}. 
This, by Lemma \ref{lemma-r}, proves that there is a unique solution $u$ 
of \eqref{Problem}, that completes the first part of the theorem. 

5. 
Finally, we prove the regularity estimates 
\eqref{evengeneralregularityestimate} and \eqref{oddgeneralregularityestimate} 
for $r \geq 2$, since the case $r=1$ follows from 
Theorem \ref{theorem-r:1}. 

The unique strong solution $z^r=(u^r,v^r,w^r)$ of 
\eqref{abstractCauchyproblem-r}, is given by
\begin{equation*}
  z^r(t)=e^{t\mathcal{A}} z^{r,0}+\int_0^t\! e^{(t-s)\mathcal{A}}F^r(s)\,ds,
\end{equation*}
and we recall the fact that $\lVert e^{t\mathcal{A}}\rVert_Z \le 1$, 
since $\mathcal{A}$ is an infinitesimal generator of a $C_0$ semigroup of contractions on $Z$.
Therefore
\begin{equation*}
  \lVert z^r(t) \rVert_Z 
  \le \lVert z^{r,0} \rVert_Z+\int_0^t \!\lVert F^r(s)\rVert_Z\,ds.
\end{equation*}
Since $v^r=\dot u^r=u^{r+1}$, $z^{r,0}=(u^{r}(0),u^{r+1}(0),0)$, and 
\begin{equation*}
  \begin{split}
    \lVert F^r(s) \rVert_Z
    =\lVert \tilde{f^r}(s)\rVert
    &=\lVert f^r(s)\|+\sum_{n=0}^{r-1}\| A u^n(0)\| |\cK^{r-n-1}(s)| 
    +\| A u^r(0)\| |\xi(s)|, 
  \end{split}
\end{equation*}
therefore we have
\begin{equation*}
  \begin{split}
    \Big((1-\kappa)\lVert u^r(t)\rVert_V^2&+ \lVert u^{r+1}(t)\rVert^2 \Big)^{1/2}\\
     &\le \big((1-\kappa)\lVert u^r(0) \rVert_V^2
       +\lVert u^{r+1}(0) \rVert^2\big)^{1/2}\\
     &\quad  +\int_0^t\! \rVert f^r(s)\rVert \ ds
        +\sum_{n=0}^{r-1}\| A u^n(0)\| \int_0^t\! |\cK^{r-n-1}(s)|\ ds \\
    &\quad  +\| A u^r(0)\| \int_0^t\! |\xi(s)|\ ds.
  \end{split}
\end{equation*}
Hence, considering the assumption that 
$\|\cK\|_{W_1^r(\IR^+)}<1$, we have, for some 
$C=C(\kappa,\|\cK\|_{W_1^r(\IR^+)},T)$,
\begin{equation} \label{Theorem4-eq1}
  \begin{split}
    \lVert u^r(t)\rVert_V+ \lVert u^{r+1}(t)\rVert
     &\le C\bigg(
       \lVert u^r(0) \rVert_V
       +\lVert u^{r+1}(0) \rVert\\
     &\quad  +\int_0^t\! \rVert f^r(s)\rVert \ ds
       +\sum_{n=0}^{r}\| A u^n(0)\| \bigg).
  \end{split}
\end{equation}
Since, by elliptic regularity estimate \eqref{ellipticreg}, 
\begin{equation*}
  \|u^r(0)\|_V \leq \|u^r(0)\|_{H^2} \leq C \|A u^r(0)\|,
\end{equation*}
so we have 
\begin{equation*} 
  \begin{split}
    \lVert u^r(t)\rVert_V+ \lVert u^{r+1}(t)\rVert
     &\le C\bigg(
       \| u^{r+1}(0) \|+\sum_{n=0}^{r}\| A u^n(0)\| +\int_0^t\! \rVert f^r(s)\rVert \ ds
        \bigg),
  \end{split}
\end{equation*}
that, by \eqref{InitialInequality-0}--\eqref{InitialInequality-odd}, 
implies the regularity estimates 
\eqref{evengeneralregularityestimate}--\eqref{oddgeneralregularityestimate}. 
Now, the proof is complete.
\end{proof}

\section{The spatial finite elment discretization}
The variational form of \eqref{Problem}  
is to find $u(t) \in V$, such that 
$u(0)=u^0$, $\dot u(0)=u^1$, and for $t \in (0,T)$,
\begin{equation}  \label{VF}
  (\ddot u,v)+a(u,v) -\int_0^t\! \cK(t-s)a(u(s),v)\ d s=(f,v),  
    \quad \forall v \in V.
\end{equation}

Let $\Omega$ be a convex polygonal domain and 
$\{\cT_h\}$ be a regular family of 
 triangulations of $\Omega$ with corresponding family of 
 finite element spaces $V_h^l \subset V$,
 consisting of continuous piecewise polynomials of degree 
 at most $l-1$, that vanish on 
$\partial\Omega$ (so the mesh is required to fit $\partial \Omega$). 
Here $l\ge 2$ is an integer number. 
We define piecewise constant mesh function 
$h_K(x) = \textrm{diam}(K)$ for $x\in K,\,K\in \mathcal{T}_h$, 
and for our error analysis we denote $h=\max_{K\in\mathcal{T}_h}h_K$. 
We note that the finite element spaces $V_h^l$ have the property that
\begin{equation} \label{GeneralErrorEstimate}
  \min_{\chi \in V_h^l}
  \{\|v-\chi\| + h\|v-\chi\|_1\}\le C h^i\|v\|_i, 
  \quad \textrm{for} \ v \in  H^i \cap V,\ 1\le i\le l.
\end{equation} 

We recall  the $L_2$-projection $\cP_h:H\to V_h^l$ and 
the Ritz projection $\cR_h:V\to V_h^l$ defined by
\begin{align*}
  a(\cR_h v, \chi)=a(v,\chi)\quad {\rm and}\quad 
  (\cP_h v,\chi)=(v,\chi),\qquad \forall \chi \in V_h^l.
\end{align*}
We also recall the elliptic regularity estimate \eqref{ellipticreg}, 
such that the error estimates \eqref{GeneralErrorEstimate} hold true 
for the Ritz projection $\cR_h$, see \cite{Thomee_Book}, i.e.,
\begin{equation} \label{errorRh}
  \begin{split}
    \|(\cR_h-I)v\| 
    + h\|(&\cR_h-I)v\|_1\le C h^i\|v\|_i, 
    \quad \textrm{for} \ v \in H^i \cap V,\ 1\le i\le l.
  \end{split}
\end{equation} 

Then, the spatial finite element discretization of \eqref{VF} 
is to find $u_h(t) \in V_h^l$ such that 
$u_h(\cdot,0)=u_h^0,\ \dot u_h(\cdot,0)=u_h^1$, and for $t\in(0,T)$,
\begin{equation}  \label{FE}
  (\ddot u_h,v_h)+a(u_h,v_h) -\int_0^t\! \cK(t-s)a(u_h(s),v_h)\ d s=(f,v_h),  
    \quad \forall v_h \in V_h^l, 
\end{equation}
where $u_h^0$ and $u_h^1$ are suitable approximations to be chosen, 
respectively, for $u_0$ and $u^1$ in $V_h^l$.

\begin{theorem} \label{Apriori-Semidiscrete}
Assume that $\Omega$ is a convex polygonal domain. 
Let $u$ and $u_h$ be, respectively, the solutions of \eqref{VF} and \eqref{FE}.
Then  
\begin{equation}   \label{A_Priori_1}
  \begin{split}
     \|u_h(T)-u(T)\|
     &\le C\|u_h^0-\cR_h u^0\|
     + Ch^l\Big(\|u(T)\|_l +\int_0^T\!\|\dot u\|_l \,d\tau\Big),
  \end{split}
\end{equation}
where we assume the initial condition $u_h^1=\cP_h u^1$. 
\end{theorem}

\begin{proof}
The proof is adapted from \cite{Baker1976}. 
We split  the error as 
\begin{equation}\label{Error_Split}
  e=u_h-u=(u_h-R_hu)+(R_hu-u)=\theta+\omega.
\end{equation}
We need to estimate $\theta$, since the spatial projection error $\omega$ 
is estimated from \eqref{errorRh}. 

So, putting $\theta$  in  \eqref{FE} we have, for $v_h \in V_h^l$, 
\begin{equation*} 
  \begin{split}
    (\ddot{\theta},v_h&)+a(\theta,v_h)-\int_0^t \cK(t-s) a(\theta(s),v_h)ds\\
    &= (\ddot u_h,v_h)+a(u_h,v_h)-\int_0^t \cK(t-s) a(u_h(s),v_h)ds\\
    &\quad - (R_h\ddot{u},v_h)-a(R_h u,v_h)+\int_0^t \cK(t-s) a(R_hu(s),v_h)ds,
  \end{split}
\end{equation*}
that, using \eqref{FE}, the definition of the Ritz projection $\cR_h$, and 
\eqref{VF}, we have 
\begin{equation}\label{theta_eq1}
  \begin{split}
    (\ddot{\theta},v_h&)+a(\theta,v_h)-\int_0^t \cK(t-s) a(\theta(s),v_h)ds\\
    &= (f,v_h)-(R_h\ddot{u},v_h)-a(u,v_h)+\int_0^t \cK(t-s) a(u(s),v_h)ds\\
    &= (\ddot{u},v_h)-(R_h\ddot{u},v_h)
    =-(\ddot{\omega},v_h).
  \end{split}
\end{equation}
Therefore we can write, for $v_h(t) \in V_h^l$, $t \in (0,T]$,
\begin{equation*}
  \frac{d}{dt}(\dot{\theta},v_h)-(\dot{\theta},\dot v_h)
    +a(\theta,v_h)- \int_0^t \cK(t-s) a(\theta(s),v_h(t))ds
   =-\frac{d}{dt}(\dot{\omega},v_h)+(\dot{\omega},\dot{v_h}),
\end{equation*}
that, recalling $e=\theta+\omega$, we obtain
\begin{equation}\label{Psi_eq2}
  -(\dot{\theta},\dot v_h)+a(\theta, v_h)- \int_0^t  \cK(t-s) a(\theta(s),v_h(t))ds
  =-\frac{d}{dt}(\dot{e},v_h)+(\dot{\omega},\dot v_h)
\end{equation}
Now let $ 0<\varepsilon \leq T$ , and we make the particular choice
\begin{equation*}
  v_h(\cdot,t)
  =\int_t^{\varepsilon}\theta(\cdot,\tau)d\tau ,\qquad  0\leq t \leq T,
\end{equation*}
then clearly we have
\begin{equation} \label{zz}
  v_h(\cdot,\varepsilon)=0,\quad \frac{d}{dt}v_h(\cdot,t)=-\theta(\cdot,t),
   \quad 0\leq t \leq T.
\end{equation}
Hence, considering \eqref{zz} in \eqref{Psi_eq2} , we have
\begin{equation*}
\frac{1}{2} \frac{d}{dt} (\|\theta \|^2-\|v_h \|_V^2)-\int_0^t 
\cK(t-s) a(\theta(s),v_h(t))ds
=-\frac{d}{dt}(\dot{e},v_h)-(\dot{\omega},\theta).
\end{equation*} 
Now, integrating from $t=0$ to $t=\varepsilon$, we have
\begin{equation*}
\begin{split}
\|\theta(\varepsilon) \|^2-\|\theta(0) &\|^2-\|v_h(\varepsilon)\|^2_V
+ \|v_h(0)\|^2_V-2\int_0^{\varepsilon}\int_0^t \cK(t-s) a(\theta(s),v_h(t))dsdt\\
&=-2(\dot{e}(\varepsilon),v_h(\varepsilon))+ 2(\dot{e}(0),v_h(0))
- 2\int_0^{\varepsilon}(\dot \omega,\theta)dt.
\end{split}
\end{equation*}
Then, using the initial assuption $u_h^1=\cP_h u^1$  that implies  
the second term on the right side is zero and recalling 
$v_h(\varepsilon)=0$, we conclude
\begin{equation}  \label{eq3}
  \begin{split}
  \|\theta(\varepsilon) \|^2 +\|&v_h(0)\|_V^2 - 2\int_0^{\varepsilon}
     \int_0^t \cK(t-s) a(\theta(s),v_h(t))dsdt \\
  &\leq \|\theta(0) \|^2+2\max_{0 \leq t \leq \varepsilon} \|\theta (t)\| 
     \int _0^{\varepsilon} \|\dot \omega\|dt.
  \end{split}
\end{equation}
Now, by changing the order of integrals, 
using $ \frac{d}{dt}\xi(t-s)=-\cK(t-s) $ from \eqref{xiproperty}, 
and integration by parts, we can write the third term on the left 
side as
\begin{equation*}
\begin{split}
  -2\int_0^{\varepsilon}\int_0^t &\cK(t-s) a(\theta(s),v_h(t))dsdt
  =2\int_0^{\varepsilon}\int_s^{\varepsilon}\frac{d}{dt}\xi(t-s) a(\theta(s),v_h(t))dtds\\
  &=2\int_0^{\varepsilon} \xi(\varepsilon -s) 
   a(\theta(s),v_h(\varepsilon))ds-2\int_0^{\varepsilon} \xi(0) a(\theta(s),v_h(s))ds\\
  &\qquad-2\int_0^{\varepsilon}\int_s^{\varepsilon}\xi(t-s)a(\theta(s),\dot v_h(t))dtds.
\end{split}
\end{equation*}
Then, using \eqref{zz} and $\xi(0)=\kappa$, we have
\begin{equation*}
  \begin{split}
    -2\int_0^{\varepsilon}&\int_0^t \cK(t-s) a(\theta(s),v_h(t))dsdt\\
    &=\kappa(\|v_h(\varepsilon)\|^2_V-\|v_h(0)\|^2_V)
   +2\int_0^{\varepsilon}\int_s^{\varepsilon} \xi(t-s) a(\theta(s),\theta(t))dtds.
  \end{split}
\end{equation*}
Therefore, using this and $v_h(\varepsilon)=0$ in \eqref{eq3} we have
\begin{equation*}  
  \begin{split}
  \|\theta(\varepsilon) \|^2 +(1-\kappa)\|&v_h(0)\|_V^2 
  +2\int_0^{\varepsilon}\int_s^{\varepsilon} \xi(t-s) a(\theta(s),\theta(t))dtds\\
  &\leq \|\theta(0) \|^2+2\max_{0 \leq t \leq \varepsilon} \|\theta (t)\| 
     \int _0^{\varepsilon} \|\dot \omega\|dt,
  \end{split}
\end{equation*}
that considering the fact that $\xi$ is a positive type kernel and $\kappa<1$, 
in a standard way, implies that
\begin{equation*}
   \|\theta(T)\|\leq C(\|\theta(0)\|+\int_0^T\|\dot{\omega}\|d\tau). 
\end{equation*}
Hence, recaling \eqref{Error_Split}, we have
\begin{equation*} 
  \begin{split}
   \|e(T)\|
   &\leq \|\theta(T)\|+\|\omega(T)\|\\
   &\leq  C\Big(\|u_h^0-R_hu^0\|+\int _0^T\|\dot u-R_h \dot u\|dt\Big)
     +\|(R_hu -u)(T)\|,
  \end{split}
\end{equation*} 
that using the error estimate \eqref{errorRh} implies the a priori error 
estimate \eqref{A_Priori_1}. 
\end{proof}

\section{The continuous Galerkin method}
Here we formulate a continuous space-time Galerkin finite 
element method of order one, 
cG(1)cG(1), for the primar and dual problems \eqref{weakformprimary} 
and \eqref{weakformdual}, that is based on a similar method for the 
wave equation in \cite{BangerthGeigerRannacher2010}.  
Then we prove stability estimaes for the 
discrete dual problem. These are then used in a priori error analysis, 
that is via duality. 

\subsection{Weak formulation}
First we write a ``velocity-displacement'' formulation of \eqref{Problem}
which is obtained by introducing a new velocity variable. 
We use the new variables $u_1=u$, $u_2=\dot u$, and 
$u=(u_1,u_2)$, 
then the variational form is to find 
$u_1(t),\,u_2(t)\in V$ such that
$u_1(0)=u^0$, $u_2(0)=v^0$, and for $t \in (0,T)$,
\begin{equation} \label{weakform}
  \begin{split}
    &(\dot u_1(t),v_1)-(u_2(t),v_1)=0,\\
    &(\dot u_2(t),v_2) + a(u_1(t),v_2)
    - \int_0^t \cK(t-s) a(u_1(s), v_2) \, ds \\
    &\qquad\qquad\qquad\quad\quad
           = (f(t),v_2),
    \quad \forall v_1,v_2 \in V .
  \end{split}
\end{equation}

Now we define the bilinear and linear forms
$B:\cU\times \cV \to \IR$ and 
$L:\cV \to \IR$  
by
\begin{align}  \label{BL}
  \begin{aligned}
    B(u,v)
    &=\int_0^T\!
      \Big\{(\dot u_1,v_1)-(u_2,v_1)
       +(\dot u_2,v_2)+a(u_1,v_2)\\
    &\quad\quad-\int_0^t\!\cK(t-s)
      a\big(u_1(s),v_2\big)\,ds\Big\}\,dt\\
    &\quad+\big(u_1(0),v_1(0)\big)+\big(u_2(0),v_2(0)\big),\\
    L(v)&=\int_0^T\!(f,v_2)\,dt
      +\big(u^0,v_1(0)\big)+\big(v^0,v_2(0)\big),
  \end{aligned}
\end{align}
where
\begin{align}   \label{UV}
  \begin{aligned}
    \cU&=H^1(0,T;V)\times H^1(0,T;H),\\
    \cV&=\big\{v=(v_1,v_2):
      w\in L_2(0,T;H)\times L_2(0,T;V),
      v_i \text{ right continuous in}\ t \big\}.
  \end{aligned}
\end{align}

We note that the weak form \eqref{weakform} can be writen as: 
find $u\in \cU$ such that,
\begin{equation} \label{weakformprimary}
  B(u,v)=L(v),\quad \forall v\in\cV.
\end{equation}
Here the definition of the velocity $u_2=\dot u_1$ is enforced in the
$L_2$ sense, and the initial data are placed in the bilinear form in
a weak sense. A variant is used in \cite{StigFardin2010} where the
velocity has been enforced in the $H^1$ sense, without placing the initial data in the bilinear form. 
We also note that the initial data  are retained by the choice of 
the function space $\cV$, that
consists of right continuous functions with respect to time.

Our a priori error analysis for the full discrete problem, 
cG(1)cG(1) method in $\S6$, is based on the duality arguments, 
and therefore we formulate 
the dual form of \eqref{weakformprimary}. 
To this end, we define the bilinear and linear forms
$B_\tau^*:\cV^*\times\cU^*\to\IR,\,
L_\tau^*:\cV^*\to\IR$,
 for $\tau\in \IR^{\ge 0}$, by
\begin{align}  \label{BLdual}
  \begin{aligned}
    B_\tau^*(v,z)
    &=\int_{\tau}^T\!
      \Big\{-(v_1,\dot z_1)+a(v_1,z_2)
      -\int_t^T\!\cK(s-t)a\big(v_1,z_2(s)\big)\,ds \\
    &\quad \quad - (v_2,\dot{z}_2)-(v_2,z_1) \Big\}\,dt
      +\big(v_1(T),z_1(T)\big)
        +\big(v_2(T),z_2(T)\big),\\
    L^*_\tau(v)
    &=\int_{\tau}^{T}\!
      \Big\{(v_1,j_1)+(v_2,j_2)\Big\}\,dt
        +\big(v_1(T),z_1^{T}\big)
        +(v_2(T),z_2^{T}),
  \end{aligned}
\end{align}
where $j_1,\,j_2$ and $z_1^{T},\,z_2^{T}$ represent, respectively,
the load terms and the initial data of the dual (adjoint) problem.
In case of $\tau=0$, we use the notation
$B^*,L^*$ for short.
Here
\begin{align}   \label{UVdual}
  \begin{aligned}
    \cU^*&=H^1(0,T;H)\times H^1(0,T;V),\\
    \cV^*&=\big\{v=(v_1,v_2) \in L_2(0,T;V)\times L_2(0,T;H):
      v_i \text{ left continuous in } t \big\}.
  \end{aligned}
\end{align}
We note that, recalling \eqref{UV},  
$\cU\subset \cV^*,\,\cU^*\subset\cV$.

We also note that $B^*$ is the adjoint form of $B$. 
Indeed, integrating by parts with respect to time 
in $B$, then changing the order of integrals in the convolution 
term as well as changing the role of the variables $s,t$, we have,
\begin{equation}   \label{BequalBstar}
  B(u,v)=B^*(u,v),
  \quad \forall u\in\cU,\, v\in\cU^*.
\end{equation}

Hence, the variational form of the dual problem
is to find $z\in \cU^*$ such that,
\begin{equation}   \label{weakformdual}
  B^*(v,z)=L^*(v),\quad \forall v\in\cV^*.
\end{equation}

\subsection{The cG(1)cG(1) method}
Let $0=t_0<t_1<\cdots<t_{n-1}<t_n<\cdots<t_N=T$ be a partition of the time interval $[0,T]$.
To each discrete time level $t_n$ we associate a triangulation $\cT_h^n$ of the polygonal domain $\Omega$ with the mesh function,
\begin{equation}   \label{hn}
  h_n(x)=h_K=\diam(K),\quad x\in K,\,K\in\cT_h^n,
\end{equation}
and a finite element space $V_h^n$ consisting of continuous
piecewise linear polynomials.
For each time subinterval $I_n=(t_{n-1},t_n)$ of length $k_n=t_n-t_{n-1}$,
we define intermediate triangulaion $\bar \cT_h^n$ which is composed
of the union of the neighboring meshes
$\cT_h^n,\,\cT_h^{n-1}$ defined at discrete time levels $t_n,\,t_{n-1}$, respectively.
The mesh function $\bar h_n$ is then defined by
\begin{equation}   \label{barhn}
  \bar h_n(x)=\bar h_K=\diam(K),\quad x\in K,\,K\in\bar\cT_h^n.
\end{equation}
Correspondingly, we define the finite element spaces
$\bar V_h^n$ consisting of continuous piecewise linear polynomials.
This construction is used in order to allow continuity in time of the
trial functions when the meshes change with time.
Hence we obtain a decomposition of each time slab
$\Omega^n=\Omega\times I_n$ into space-time cells
$K\times I_n,\,K\in \bar\cT_h^n$
(prisms, for example, in case of $\Omega\subset \IR^2$).
The trial and test function spaces for the discrete form are,
respectively:
\begin{equation}   \label{discretespaces}
  \begin{split}
    \cU_{hk}&=\Big\{U=(U_1,U_2):
      U \text{ continuous in } \Omega\times [0,T], 
        U(x,t)|_{I_n} \text{ linear in } t,\\
      &\qquad\qquad\qquad\qquad\
        U(\cdot,t_n)\in (V_h^n)^2, U(\cdot,t)|_{I_n}
          \in (\bar V_h^n)^2 \Big\},\\
    \cV_{hk}&=\Big\{V=(V_1,V_2):
      V(\cdot,t) \text{ continuous in } \Omega, 
         V(\cdot,t)|_{I_n}\in (V_h^n)^2,\\
       &\qquad\qquad\qquad\qquad\
         V(x,t)|_{I_n} \text{ piecewise constant in } t\Big\}.
  \end{split}
\end{equation}
We note that global continuity of the trail functions in $\cU_{hk}$
requires the use of `\textit{hanging nodes}' if the spatial mesh changes
across a time level $t_n$.
We allow one hanging node per edge or face.

\begin{rem}  \label{MeshRefind}
If we do not change the spatial mesh or just refine the
spatial mesh from one time level to the next one, i.e.,
\begin{equation}   \label{assumption}
  V_h^{n-1}\subset V_h^n,\quad n=1,\dots,N,
\end{equation}
then we have $\bar V_h^n=V_h^n$.
\end{rem}

In the construction of $\cU_{hk}$ and $\cV_{hk}$ we have associated
the triangulation $\cT_h^n$ with discrete time levels instead of the time
slabs $\Omega^n$, and in the interior of time slabs we let $U$ be from the
union of the finite element spaces defined on the triangulations at the
two adjacent time levels.
This construction is necessary to allow for trial functions that are continuous also at the discrete time leveles even if grids change between
time steps. 
For more details and computational aspects, including hanging nodes, 
see \cite{FardinarXiv:1205.0159} and the references therein. 
Associating triangulation with time slabs instead of time levels would
yield a variant scheme which includes jump terms due to discontinuity
at discrete time leveles, when coarsening happens.
This means that there are extra degrees of freedom that one might use
suitable projections for transfering solution at the time levels $t_n$,
 see \cite{StigFardin2010}.

The continuous Galerkin method, based on the variational formulation \eqref{weakform}, is to find $U\in\cU_{hk}$ such that,
\begin{equation}   \label{cG}
  B(U,V)=L(V),\quad \forall\, V\in \cV_{hk}.
\end{equation}
Here, as a natural choice, we consider the initial conditions 
\begin{equation}   \label{initialdata}
  u_h^0:=U_1(0)=\cP_h u^0,\quad u_h^1:=U_2(0)=\cP_h u^1,
\end{equation}
where the $L_2$ pojection $\cP_h$ is defined in \eqref{projections}. 

The Galerkin orthogonality, with $u=(u_1,u_2)$ being the exact solution
of \eqref{weakform}, is then,
\begin{equation}   \label{Galerkinorthogonality}
  B(U-u,V)=0,\quad \forall\, V\in \cV_{hk}.
\end{equation}
Similarly the continuous Galerkin method, based on the dual variational formulation \eqref{weakformdual}, is to find $Z\in\cU_{hk}$ such that,
\begin{equation}   \label{dualcG}
  B^*(V,Z)=L^*(V),\quad \forall\, V\in \cV_{hk}.
\end{equation}
Then, $Z$ also satisfies, for $n=0,1,\dots,N-1$,
\begin{equation}   \label{dualcG2}
  B_{t_n}^*(V,Z)=L^*_{t_n}(V),\quad \forall\, V\in \cV_{hk}.
\end{equation}

From \eqref{cG} we can recover the time stepping scheme,
\begin{equation}   \label{cG2}
  \begin{split}
    &\int_{I_n} \!\Big\{(\dot{U}_1,V_1)-(U_2,V_1) \Big\}\,dt=0,\\
    &\int_{I_n}\!\Big\{(\dot{U}_2,V_2)+a(U_1,V_2)
       -\int_0^t \! \cK(t-s)a\big(U_1(s),V_2\big)\,ds\Big\}\,dt\\
    &\qquad\qquad
      =\int_{I_n}\! (f,V_2)\,dt,
      \quad\forall\, (V_1,V_2)\in \cV_{hk},\\
    &U_1(0)=u_h^0,\quad U_2(0)=u_h^1,
  \end{split}
\end{equation}
with the initial conditions \eqref{initialdata}. 

Typical functions
$U=(U_1,U_2)\in\cV_{hk},\,W=(W_1,W_2)\in\cW_{hk}$ are as follows:
\begin{equation}   \label{typicalfunctions}
  \begin{split}
    U_i(x,t_n)&=U_i^n(x)=\sum_{j=1}^{m_n} U_{i,j}^n \vp_j^n(x),\\
    U_i(x,t)|_{I_n}&=\psi_{n-1}(t)U_i^{n-1}(x)+\psi_n(t)U_i^n(x),\\
    W_i(x,t)|_{I_n}&=\sum_{j=1}^{m_n} W_{i,j}^n \vp_j^n(x),
  \end{split}
\end{equation}
where $m_n$ is the number of degrees of freedom in $\cT_h^n$,
$\{\vp_j^n(x)\}_{j=1}^{m_n}$ are the nodal basis functions for
$V_h^n$ defined on triangulation $\cT_h^n$, and $\psi_n(t)$ is
the nodal basis function defined at time level $t_n$.
Hence \eqref{cG2} yields
\begin{equation*}
  \begin{split}
    &M^n U_1^n - \frac{k_n}{2}M^n U_2^n
      =M^{n-1,n} U_1^{n-1}
      +\frac{k_n}{2}M^{n-1,n} U_2^{n-1},\\
    & M^n U_2^n
      +(\frac{k_n}{2}-\omega_{n,n}^-)S^n U_1^n
      = M^{n-1,n} U_2^{n-1}
      +(-\frac{k_n}{2}+\omega_{n,n}^+)S^{n-1,n} U_1^{n-1}\\
     &\qquad\qquad\qquad\qquad\qquad\qquad\qquad
      +\sum_{l=1}^{n-1}\Big(\omega_{n,l}^-S^{l,n} U_1^l
                                   +\omega_{n,l}^+S^{l-1,n} U_1^{l-1}\Big)+B^n,\\
    & U_1^0=u_h^0,\quad U_2^0=u_h^1,
  \end{split}
\end{equation*}
where, for $l=1,\dots,n$,
\begin{align*}
  \begin{aligned}
    \omega_{n,l}^+&=\int_{I_n}\!
      \int_{t_{l-1}}^{t\wedge t_l}\!
      \cK(t-s)\psi_{l-1}(s)\,ds\,dt,
      \qquad 
      \omega_{n,l}^-=\int_{I_n}\!
      \int_{t_{l-1}}^{t\wedge t_l}\!
      \cK(t-s)\psi_l(s)\,ds\,dt,\\
    B^n&=(B_i^n)_i=\Big(\int_{I_n}\!(f,\vp_i)\, dt\Big)_{\!i}\,,\\
    M^n&=(M_{ij}^n)_{ij}=\big((\vp_j^n,\vp_i^n)\big)_{ij},\quad
      M^{n-1,n}=(M_{ij}^{n-1,n})_{ij}
      =\big((\vp_j^{n-1},\vp_i^n)\big)_{ij},\\
    S^{l,n}&=(S_{ij}^{l,n})_{ij}=\big(a(\vp_j^l,\vp_i^n)\big)_{ij}.
  \end{aligned}
\end{align*}

We define the orthogonal projections $\cR_{h,n}:V\to V_h^n$,
$\cP_{h,n}:H\to V_h^n$ and $\cP_{k,n}:L_2(I_n)^d\to\IP^d_0(I_n)$,
respectively, by
\begin{align}  \label{projections}
  \begin{aligned}
    a(\cR_{h,n}v-v,\chi)&=0,&&\forall v\in V,\,
      \chi\in V_h^n,\\
    (\cP_{h,n}v-v,\chi)&=0,&&\forall v\in H,\,
      \chi\in V_h^n,\\
    \int_{I_n}\!\!(\cP_{k,n}v-v)\cdot \psi\, dt&=0,&&
       \forall  v\in L_2(I_n)^d, \,\psi\in \IP^d_0(I_n)\,,
  \end{aligned}
\end{align}
with $\IP_0^d$ denoting the set of all vector-valued constant
polynomials. Correspondingly, we define $\cR_hv,\,\cP_hv$
and $\cP_kv$ for
$t\!\in\! I_n\,(n=1,\cdots,N)$, by
$(\cR_hv)(t)=\cR_{h,n}v(t),\,
(\cP_hv)(t)=\cP_{h,n}v(t)$, and
$\cP_kv=\cP_{k,n}(v\!\!\mid_{I_n})$.
\begin{rem}   \label{PkV}
In the case of assumption \eqref{assumption},
by Remark \ref{MeshRefind} and the definition of the
$L_2$-projection $\cP_k$, we have
$\dot V,\,\cP_k V \in \cV_{hk}$, for any $V\in\cU_{hk}$.
\end{rem}

We introduce the discrete linear operator
$A_{n,r}:V_h^r\to V_h^n$ by
\begin{equation*}
  a(v_r,w_n)
   =(A_{n,r}v_r,w_n),\quad\forall v_r\in V_h^r,
    \,w_n \in V_h^n.
\end{equation*}
We set $A_n=A_{n,n}$, with discrete norms
\begin{equation*}
  \| v_n\|_{h,l}=\| A_{n}^{l/2} v_n\|
     =\sqrt{(v_n,A_n^l v_n)},\,\quad v_n\in V_h^n\,\,
       \textrm{and}\,\, l\in \IR\,,
\end{equation*}
and $A_h$ so that $A_h v=A_n v$ for $v \in V_h^n$.
We use $\bar A_h$ when it acts on $\bar V_h^n$.
For later use in our error analysis we note that
$\cP_h A=A_h \cR_h$.

\subsection{Stability of the solution of the discrete dual problem}
We know that stability estimates and the corresponding analysis for 
dual problem is similar to the primar problem, however with opposite 
time direction. 
Hence, having a smooth or weakly singular kernel with 
\eqref{KernelProperty}, we can quote slightly different energy identities, 
compare to \eqref{discretestabilitydual}, from \cite{StigFardin2010} or \cite{AdolfssonEnelundLarssonRacheva2006} 
for the discrete dual solution, from which similar stability estimates to 
\eqref{discretestabilitydualineq} is obtained, 
though with different projections and constants.

To prove stability estimates in \cite{StigFardin2010} and 
\cite{FardinBIMS2012} we have used auxiliary functions 
in the form, respectively, 
$W(t,s)=U(t)-U(s)$ and $W(t,s)=U(t)-U(t-s)$. 
Here, using the properties of the fuction $\xi=\xi(t)$ in the convolution 
integral and partial integration, we give a proof  
which is straightforward. 

We note that the stability constant in \eqref{discretestabilitydualineq} 
does not depend on $t$. See \cite{RiviereShawWhiteman2007},  
\cite{ShawWhiteman2004} and \cite{RiveraMenzala1999}, 
where stability estimates have been represented, 
in which the stability factor depends on $t$, due to Gronwall's lemma.  
\begin{theorem}
Let $Z$ be the solution of \eqref{dualcG} with
sufficiently smooth data $z_1^T,z_2^T$, $j_1,j_2$. 
Further, we assume \eqref{assumption}.
Then for $l \in \IR$, we have the identity,
\begin{equation}   \label{discretestabilitydual}
  \begin{split}
    \|Z_1(t_n)&\|_{h,l}^2
     +\tilde{\kappa} \|Z_2(t_n)\|_{h,l+1}^2
     +2\int_{t_n}^T\!\int_t^T\!
      \xi(s-t)a\big(A_h^l\dot{Z}_2(t),\dot Z_2(s)\big)\,ds\,dt \\
    &=\|Z_1(T)\|_{h,l}^2+(1+\kappa) \|Z_2(T)\|_{h,l+1}^2\\
    &\quad +2\int_{t_n}^T\!(A_h^l Z_1,\cP_k\cP_hj_1)\,dt
     +2\int_{t_n}^T\!
      (A_h^{l+1}Z_2,\cP_k\cP_hj_2)\,dt\\
    &\quad-2\int_{t_n}^T\!\int_t^T\!
     \cK(s-t)a\big(A_h^l\cP_k\cP_hj_2,Z_2(s)\big)\,ds\,dt\\
    &\quad-2\int_{t_n}^T\!
     \cK(T-t)a\big(A_h^l Z_2(t),Z_2(T)\big)\,dt\\
    &\quad-2\xi(T-t_n)a\big(A_h^lZ_2(t_n),Z_2(T)\big),
  \end{split}
\end{equation}
where $\tilde{\kappa}=1-\kappa$. 
Moreover, for some constant $C=C(\kappa)$, we have stability estimate
\begin{equation}   \label{discretestabilitydualineq}
  \begin{split}
    \|Z_1(t_n)\|_{h,l}+\|Z_2(t_n)\|_{h,l+1}
    &\le C\Big\{\| Z_1(T)\|_{h,l}+\| Z_2(T)\|_{h,l+1}\\
    &\qquad\  +\int_{t_n}^T\!
      \Big(\|\cP_h j_1\|_{h,l}
      +\|\cP_h j_2\|_{h,l+1} \Big)\,dt\Big\}.
  \end{split}
\end{equation}
Here, we set the initial data of the dual problem as 
\begin{equation} \label{initialdatadual}
    Z_i(T)=\cP_h z_i^T, \quad i=0,1.
\end{equation}
\end{theorem}
\begin{proof}
1. The solution $Z$ of \eqref{dualcG} also satisfies \eqref{dualcG2}, for $n=N-1,\dots,1,0$. Then recalling Remark \ref{PkV} for the assumption \eqref{assumption}, we obviously have,
\begin{equation}   \label{Z1Z2}
  \cP_kZ_1=-\dot Z_2-\cP_k\cP_hj_2.
\end{equation}

2. Using this in \eqref{dualcG2} we obtain
\begin{equation*}
  \begin{split}
    \int_{t_n}^T\!\Big\{-(&V_1,\dot Z_1)+a(V_1,Z_2)
     -\int_t^T\!\cK(s-t)a\big(V_1,Z_2(s)\big)\,ds\Big\}\,dt
     +\big(V_1(T),Z_1(T)\big)\\
     &   +\big(V_2(T),Z_2(T)\big)
     =\int_{t_n}^T\!(V_1,j_1)\,dt
      +\big(V_1(T),z_1^T\big)
     +\big(V_2(T),z_2^T\big).
  \end{split}
\end{equation*}
For the convolution term we recall
$\cK(s-t)=-D_s\xi(s-t)$ from \eqref{xiproperty}, and then
partial integration yields
\begin{equation*}
  \begin{split}
    -\int_{t_n}^T\! 
     &\int_t^T\! \cK(s-t)a\big(V_1,Z_2(s)\big)\,ds\,dt\\
     &= -\int_{t_n}^T\!\int_t^T\!
     \xi(s-t)a\big(V_1,\dot Z_2(s)\big)\,ds\,dt
     +\int_{t_n}^T\!\xi(T-t)a\big(V_1,Z_2(T)\big)\,dt \\
     &\qquad -\kappa \int_{t_n}^T\!a\big(V_1,Z_2(t)\big)\,dt.
 \end{split}
\end{equation*}
These and $\tilde\kappa=1-\kappa$ imply that the solution $Z$ satisfies,
\begin{equation*}
  \begin{split}
    \int_{t_n}^T\!\Big\{-(V_1,\dot Z_1)
     &+\tilde \kappa a(V_1,Z_2)
      -\int_t^T\!\xi(s-t)a\big(V_1,\dot Z_2(s)\big)\,ds\\
     & +\xi(T-t)a(V_1,Z_2(T))\Big\}\,dt 
     +\big(V_1(T),Z_1(T)\big)
     +\big(V_2(T),Z_2(T)\big)\\
     &\!\!\!\!
      =\int_{t_n}^T\!(V_1,\cP_hj_1)\,dt+\big(V_1(T),z_1^T\big)
     +\big(V_2(T),z_2^T\big),
  \end{split}
\end{equation*}
Now we set $V_i=A_h^l\cP_kZ_i$, and recall the initial data 
\eqref{initialdatadual} such that the terms concerning the initial data 
are canceled by the definition of the orthogonal projection 
$\cP_h$. 
Then  we have
\begin{equation}   \label{stability:eq0}
  \begin{split}
    \int_{t_n}^T\!\Big\{&-(A_h^l\cP_kZ_1,\dot Z_1)
      +\tilde \kappa a(A_h^l\cP_kZ_1,Z_2)
      -\int_t^T\!\xi(s-t)a\big(A_h^l\cP_kZ_1,\dot Z_2(s)\big)\,ds\\
     & +\xi(T-t)a\big(A_h^l\cP_kZ_1,Z_2(T)\big)\Big\}\,dt
      =\int_{t_n}^T\!(A_h^l\cP_kZ_1,\cP_hj_1)\,dt.
  \end{split}
\end{equation}

3. We study the four terms at the left side of the above equation. 
For the first term we have
\begin{equation}   \label{stability:eq1}
    \int_{t_n}^T\!-(A_h^l\cP_kZ_1,\dot Z_1)\,dt
    =-\frac{1}{2}\|Z_1(T)\|_{h,l}^2
     +\frac{1}{2}\|Z_1(t_n)\|_{h,l}^2.
\end{equation}
With \eqref{Z1Z2} we can write the second term as
\begin{equation}   \label{stability:eq2}
  \begin{split}
    \tilde \kappa \int_{t_n}^T\! &a(A_h^l\cP_kZ_1,Z_2)\,dt\\
    &=-\tilde \kappa \int_{t_n}^T\!a(A_h^l\dot Z_2,Z_2)\,dt
      -\tilde \kappa \int_{t_n}^T\!
     a(A_h^l\cP_k\cP_hj_2,Z_2)\,dt\\
    &=-\frac{\tilde \kappa}{2}\|Z_2(T)\|_{h,l+1}^2
     +\frac{\tilde \kappa}{2}\|Z_2(t_n)\|_{h,l+1}^2
     -\tilde \kappa \int_{t_n}^T\!
     a(A_h^l\cP_k\cP_hj_2,Z_2)\,dt.
  \end{split}
\end{equation}
For the third term in \eqref{stability:eq0},
by virtue of \eqref{Z1Z2} and
integration by parts, we obtain
\begin{equation}   \label{stability:eq3}
  \begin{split}
    -&\int_{t_n}^T\!\int_t^T\!
     \xi(s-t)a\big(A_h^l\cP_kZ_1,\dot Z_2(s)\big)\,ds\,dt\\
    &=\int_{t_n}^T\!\int_t^T\!
     \xi(s-t)a\big(A_h^l\dot Z_2(t),\dot Z_2(s)\big)\,ds\,dt\\
    &\quad +\int_{t_n}^T\!\int_t^T\!
     \cK(s-t)a\big(A_h^l\cP_k\cP_hj_2,Z_2(s)\big)\,ds\,dt\\
    &\quad +\int_{t_n}^T\!
     \xi(T-t)a(A_h^l\cP_k\cP_hj_2,Z_2(T))\,dt
     -\kappa \int_{t_n}^T\!
     a\big(A_h^l\cP_k\cP_hj_2,Z_2(t)\big)\,dt.
  \end{split}
\end{equation}
Finally, for the last term at the left side of \eqref{stability:eq0}, we use \eqref{Z1Z2} and integration by
parts to have
\begin{equation}   \label{stability:eq4}
  \begin{split}
    &\int_{t_n}^T\!\xi(T-t)a\big(A_h^l\cP_kZ_1,Z_2(T)\big)\,dt\\
     &\ \ =\int_{t_n}^T\!
      \cK(T-t)a\big(A_h^l Z_2(t),Z_2(T)\big)\,dt 
       -\kappa\|Z_2(T)\|_{h,l+1}^2\\
     &\ \ \quad +\xi(T-t_n)a\big(A_h^lZ_2(t_n),Z_2(T)\big)
     -\int_{t_n}^T\!
       \xi(T-t)a\big(A_h^l\cP_k\cP_hj_2,Z_2(T)\big)\,dt.
   \end{split}
\end{equation}
Putting \eqref{stability:eq1}--\eqref{stability:eq4} in
\eqref{stability:eq0} we conclude the identity
\eqref{discretestabilitydual}.

4. Now we prove the estimate \eqref{discretestabilitydualineq}.
We recall, from \eqref{positivetype}, that $\xi$ is a positive type
kernel. 
Then, using the Cauchy-Schwarz inequality in  
\eqref{discretestabilitydual} and $\|\cK\|_{L_1(\IR^+)}=\kappa$, 
$\xi(t)\le \kappa$,  we get, for $C_3=C_3(\kappa)$ and $C_4=C_4(\kappa)$, 
\begin{equation*} 
  \begin{split}
    \|Z_1(t_n)&\|_{h,l}^2
     +\tilde{\kappa} \|Z_2(t_n)\|_{h,l+1}^2\\
    &\le \|Z_1(T)\|_{h,l}^2
     +(1+\kappa) \|Z_2(T)\|_{h,l+1}^2\\
    &\quad +C_1 \max_{t_n\le t\le T}\|Z_1\|_{h,l}^2
     +1/C_1 \Big(\int_{t_n}^T\!
      \|\cP_k\cP_hj_1\|_{h,l}\,dt\Big)^2\\
    &\quad +C_2 \max_{t_n\le t\le T}\|Z_2\|_{h,l+1}^2
     +1/C_2 \Big(\int_{t_n}^T\!
      \|\cP_k\cP_hj_2\|_{h,l+1}\,dt\Big)^2\\
    &\quad +C_3 \|Z_2(T)\|_{h,l+1}^2
      +1/C_3 \max_{t_n\le t\le T}\|Z_2\|_{h,l+1}^2\\
    &\quad +C_4 \|Z_2(T)\|_{h,l+1}^2
      +1/C_4 \|Z_2(t_n)\|_{h,l+1}^2.
  \end{split}
\end{equation*}
Using that, for piecewise linear functions, we have
\begin{equation}   \label{maxU}
  \max_{[0,T]}|U_i|\le \max_{0\le n\le N}|U_i(t_n)|,
\end{equation}
and
\begin{equation*} 
  \int_0^T\!|\cP_k f|\,dt \le \int_0^T\!|f|\,dt,
\end{equation*}
and that the above inequality holds for arbitrary $N$, in a standard way,
we conclude the estimate inequality \eqref{discretestabilitydualineq}.
Now the proof is complete.
\end{proof}

\section{A priori error estimation}
We define the standard interpolant $I_k$ with $I_kv$ belong to the space
of continuous piecewise linear polynomials, and
\begin{equation}\label{definitionIk}
  I_kv(t_n)=v(t_n),\quad n=0,\,1,\,\cdots,\,N.
\end{equation}
By standard arguments in approximation theory we see that, for $q=0,\,1$,
\begin{equation}\label{errorIk}
 \int_0^T\!\!\| I_{k}v-v \|_i \,dt
  \le C k^{q+1}\int_0^T\!\!\| D_t^{q+1}v \|_i\,dt,
   \quad \textrm{for}\,\,i=0,\,1,
\end{equation}
where $ k=\max_{1\le n\le N} k_n$.

We recall that we must specialize to the pure Dirichlet boundary 
condition and a convex polygonal domain to have the elliptic 
regularity \eqref{ellipticreg}, from which the error estimates 
\eqref{errorRh} hold true for the Ritz projections in \eqref{projections}. 
We note that the energy norm $\|\cdot\|_V$ is equivalent to $\|\cdot\|_1$
on $V$.

\begin{lemma} \label{LemmaBstarB}
Assume \eqref{assumption}. Then, for $V,W \in \cU_{hk}$, we have 
\begin{equation} \label{BstarB}
  \begin{split}
    B^*(\cP_k V,W)
    &=B(V,\cP_k W)\\
    &\quad 
      +\big(V_1(0),(W_1-\cP_k W_1)(0)\big)
      +\big(V_2(0),(W_2-\cP_k W_2)(0)\big)\\
    &\quad 
      -\big(V_1(T),W_1(T)\big)
      -\big(V_2(T),W_2(T)\big) \\
    &\quad 
      +\big((\cP_k V_1)(T),W_1(T)\big)
      +\big((\cP_k V_2)(T),W_2(T)\big).
  \end{split}
\end{equation}
\end{lemma}
\begin{proof}
We recall Remark 4 for the assumption \eqref{assumption}, 
and the definition of the bilinear forms $B,B^*$ from \eqref{BL} and 
\eqref{BLdual}. Then by the definition of $\cP_k$ and 
partial integration in time we have
\begin{equation*}
  \begin{split}
    B^*&(\cP_k V,W)\\
    &=\int_0^T\!\Big\{
      -(V_1,\dot W_1)+a(V_1, \cP_k W_2)
      -\int_t^T\! \cK(s-t)a(V_1, \cP_k W_2(s))\ ds \\
    &\qquad\quad\ \  -(V_2,\dot W_2)-(V_2,\cP_k W_1) \Big\}\ dt \\
    &\quad 
      +\big((\cP_k V_1)(T),W_1(T)\big)
      +\big((\cP_k V_2)(T),W_2(T)\big)\\
    &=\int_0^T\!\Big\{
      (\dot V_1, W_1)+a(V_1, \cP_k W_2)
      -\int_t^T\! \cK(s-t)a(V_1, \cP_k W_2(s))\ ds \\
    &\qquad\quad\ \  (\dot V_2,W_2)-(V_2,\cP_k W_1) \Big\}\ dt \\
    &\quad 
      +\big(V_1(0),W_1(0)\big)
      +\big(V_2(0),W_2(0)\big)
      -\big(V_1(T),W_1(T)\big)
      -\big(V_2(T),W_2(T)\big) \\
    &\quad 
      +\big((\cP_k V_1)(T),W_1(T)\big)
      +\big((\cP_k V_2)(T),W_2(T)\big),
  \end{split}
\end{equation*}
that implies \eqref{BstarB}, and the proof is complete.
\end{proof}
\begin{theorem}
Assume that $\Omega$ is a convex polygonal domain,
and \eqref{assumption}.
Let $u$ and $U$ be the solutions of \eqref{weakformprimary} and
\eqref{cG}.
Then, with $e=U-u$ and $C=C(\kappa)$, we have
\begin{equation} \label{apriorierror1}
  \begin{split}
    \|e_1(T)\|
       &\le C h^{2}\Big(\|u^0\|_2+\|u_1(T)\|_{2}
         +\int_0^T\!\|\dot u_1\|_2 \,dt\Big)\\
       &\quad +C k^{2}\!\int_0^T\!\big(\| \ddot{u}_2\|
             +\| \ddot{u}_1\|_1 \big)\,dt,
  \end{split}
\end{equation}
and, with a quasi-uniform family of triangulations, 
\begin{equation} \label{apriorierror2}
  \begin{split}
     \|e_1(T)\|_1
       &\le C h\Big(\|u^0\|_1+\|u_1(T)\|_2
         +\int_0^T\!\|\dot u_1\|_2 \,dt\Big)\\
       &\quad +C k^{2}\!\int_0^T\!\big(\| \ddot{u}_2\|_1
             +\| \ddot{u}_1\|_2 \big)\,dt ,
  \end{split}
\end{equation}
\begin{equation} \label{apriorierror3}
  \begin{split}
      \|e_2(T)\|
       &\le C h\Big(\|u^0\|_2+\|u_2(T)\|_{1}
         +\int_0^T\!\|\dot u_1\|_2 \,dt\Big)\\
       &\quad +C k^2\!\int_0^T\!\big(\| \ddot{u}_2\|_1
             +\| \ddot{u}_1\|_2 \big)\,dt.
  \end{split}
\end{equation}

\end{theorem}
\begin{proof}
1. We recall Remark \ref{PkV} for the assumption \eqref{assumption}.
We set 
\begin{equation} \label{errorsplit}
  e=U-u=(U-I_k\pi_h u)+(I_k\pi_h u-\pi_h u)+(\pi_h u-u)=\theta+\eta+\omega,
\end{equation}
where $I_k$ is the linear interpolant defined by \eqref{definitionIk}, 
and $\pi_h$ is in terms of the projectors 
$\cR_h$ and $\cP_h$, such that 
\begin{equation} \label{choice}
  \begin{aligned}
    &\theta_1=U_1-I_k\cR_hu_1,
    &\eta_1=(I_k-I)\cR_hu_1,
    &&\omega_1=(\cR_h-I)u_1,\\
    &\theta_2=U_2-I_k\cP_h u_2,
    &\eta_2=(I_k-I)\cP_h u_2, 
    &&\omega_2=(\cP_h-I)u_2.
  \end{aligned}
\end{equation}
We note that $\eta$ and $\omega$ can be estimated by 
\eqref{errorIk} and \eqref{errorRh}, and therefore we need to estimate $\theta$. 

2. Now, putting $V=\cP_k\theta$ in \eqref{dualcG} with $j_1=j_2=0$, 
we have
\begin{equation} \label{apriori:eq1}
  \begin{split}
    L^*(\cP_k\theta)
    &=\big((\cP_k\theta_1)(T),z_1^T\big)
      +\big((\cP_k\theta_2)(T),z_2^T\big)
    =B^*(\cP_k\theta,Z), 
  \end{split}
\end{equation}
that, using Lemma \ref{LemmaBstarB} and the initial data \eqref{initialdatadual}, 
implies
\begin{equation*}
  \begin{split}
    \big(\theta_1(T),&Z_1(T)\big)
     +\big(\theta_2(T),Z_2(T)\big)\\
    &=B(\theta,\cP_k Z)
      +\big(\theta_1(0),(Z_1-\cP_k Z_1)(0)\big)
      +\big(\theta_2(0),(Z_2-\cP_k Z_2)(0)\big)\\
    &=\int_0^T\!
       \Big\{(\dot \theta_1,\cP_kZ_1)
      -(\theta_2,\cP_kZ_1)
      +(\dot \theta_2,\cP_kZ_2)
      +a(\theta_1,\cP_kZ_2)\\
    &\quad\quad -\int_0^t\!\cK(t-s)
    a\big(\theta_1(s),\cP_kZ_2\big)\,ds\Big\}\,dt\\
    &\quad +\big(\theta_1(0),Z_1(0)\big)
      +\big(\theta_2(0),Z_2(0)\big). 
  \end{split}
\end{equation*}
Then, using $\theta=e-\eta-\omega$ and the Galerkin orthogonality 
\eqref{Galerkinorthogonality}, we have 
\begin{equation*}
  \begin{split}
    \big(\theta_1(T),&Z_1(T)\big)
     +\big(\theta_2(T),Z_2(T)\big)\\
    &=\int_0^T\!
       \Big\{-(\dot \eta_1,\cP_kZ_1)
      +(\eta_2,\cP_kZ_1)
      -(\dot \eta_2,\cP_kZ_2)
      -a(\eta_1,\cP_kZ_2)\\
    &\quad\quad\quad  +\int_0^t\!\cK(t-s)
    a\big(\eta_1(s),\cP_kZ_2\big)\,ds\Big\}\,dt\\
    &\quad\quad -\big(\eta_1(0),Z_1(0)\big)
      -\big(\eta_2(0),Z_2(0)\big)\\
    &\quad +\int_0^T\!
       \Big\{-(\dot \omega_1,\cP_kZ_1)
      +(\omega_2,\cP_kZ_1)
      -(\dot \omega_2,\cP_kZ_2)
      -a(\omega_1,\cP_kZ_2)\\
    &\quad\quad\quad  +\int_0^t\!\cK(t-s)
    a\big(\omega_1(s),\cP_kZ_2\big)\,ds\Big\}\,dt\\
    &\quad\quad -\big(\omega_1(0),Z_1(0)\big)
      -\big(\omega_2(0),Z_2(0)\big). 
  \end{split}
\end{equation*}
By the definition of $\eta$, that indicates the temporal 
interpolation error, terms including $\dot\eta_i$ and $\eta_i(0)$ vanish.
We also use the definition of $\omega$ in \eqref{choice}, 
that indicates the spatial projection error, and we conclude
\begin{equation*}
  \begin{split}
    \big(\theta_1(T),&Z_1(T)\big)
     +\big(\theta_2(T),Z_2(T)\big)\\
    &=\int_0^T\!
       \Big\{(\eta_2,\cP_kZ_1)
      -a(\eta_1,\cP_kZ_2)+\int_0^t\!\cK(t-s)
    a\big(\eta_1(s),\cP_kZ_2\big)\,ds\Big\}\,dt\\
    &\quad -\int_0^T\!  (\dot \omega_1,\cP_kZ_1)\,dt
                 -\big(\omega_1(0),Z_1(0)\big),
  \end{split}
\end{equation*}
that, setting the initial data $Z_1(T)=\cP_h z_1^T=A_h^{-l}\theta_1(T)$ and  
$Z_2(T)=\cP_h z_2^T=A_h^{-(l+1)}\theta_2(T)$, $ l \in \IR$ and using the Cauchy-Schwarz 
inequality, we have
\begin{equation} \label{thetaineq-1}
  \begin{split}
    \|&\theta_1(T)\|_{h,-l}^2+ \|\theta_2(T)\|_{h,-(l+1)}^2\\
    &\le
    C_1\max_{0\le t\le T}\|\cP_kZ_1\|_{h,l}^2
    +1/C_1\Big(\int_0^T\!\|\eta_2\|_{h,-l}\,dt\Big)^2\\
    &\quad +C_2\max_{0\le t\le T}\|\cP_kZ_2\|_{h,l+1}^2
    +1/C_2\Big(\int_0^T\!\|\eta_1\|_{h,-l+1}\,dt\Big)^2\\
    &\quad +C_3\max_{0\le t\le T}\|\cP_kZ_2\|_{h,l+1}^2
    +1/C_3\Big(\int_0^T\!
    \big(\cK*\|\eta_1\|_{h,-l+1}\big)(t)\,dt\Big)^2\\
    &\quad +C_4\max_{0\le t\le T}\|\cP_kZ_l\|_{h,l}^2
    +1/C_4\Big(\int_0^T\!\| \cP_h \dot\omega_1\|_{h,-l}\,dt\Big)^2\\
    &\quad +C_5\|Z_1(0)\|_{h,l}^2
    +1/C_5\| \cP_h \omega_1(0)\|_{h,-l}^2.
  \end{split}
\end{equation}

On the other hand, putting the initial data 
$Z_1(T)=A_h^{-l}\theta_1(T)$ and  
$Z_2(T)=A_h^{-(l+1)}\theta_2(T)$ in the stability estimate 
\eqref{discretestabilitydualineq} with $j_1=j_2=0$, we have 
\begin{equation*}   
  \begin{split}
    \|Z_1(t_n)\|_{h,l}+\|Z_2(t_n)\|_{h,l+1}
    &\le C\big\{\| \theta_1(T)\|_{h,-l}+\| \theta_2(T)\|_{h,-(l+1)}\big\}.
  \end{split}
\end{equation*}
Using this, together with \eqref{maxU}, and $\|\cK\|_{L_1(\IR^+)}=\kappa$ 
in \eqref{thetaineq-1}, in a standard way, we have
\begin{equation} \label{thetaineq-2}
  \begin{split}
    \|&\theta_1(T)\|_{h,-l}+ \|\theta_2(T)\|_{h,-(l+1)}\\
    &\le C \Big\{\| \cP_h \omega_1(0)\|_{h,-l}
      +\int_0^T\!\Big(
       \| \eta_2\|_{h,-l}+\|\eta_1\|_{h,-l+1}
       +\| \cP_h \dot \omega_1\|_{h,-l}\Big)\,dt \Big\}.
  \end{split}
\end{equation}

3. To prove the first error estimate 
\eqref{apriorierror1}, we set $l=0$, and we recall the facts that 
$\|\cdot\|_{h,0}=\|\cdot\|$, $\|\cdot\|_{h,1}=\|\cdot\|_1$. 
Then, recalling $e(T)=\theta(T)+\eta(T)+\omega(T)=\theta(T)+\omega(T)$, 
and $L_2$-stability of the projection $\cP_h$, we have
\begin{equation*}
  \begin{split}
    \|e_1(T)\|
    &\le C \Big\{\|(\cR_h-I)u^0\| + \|(\cR_h-I)u_1(T)\|\\
    &\quad +\int_0^T\!\Big(
       \|(I_k-I)u_2\|+\|(I_k-I)u_1\|_1
       +\|(\cR_h-I)\dot u_1\|\Big)\,dt \Big\}.
  \end{split}
\end{equation*}
This completes the proof of the first  a priori error estimate 
\eqref{apriorierror1} by \eqref{errorIk} and \eqref{errorRh}.

4. Now, to prove the last two error estimates 
\eqref{apriorierror2}--\eqref{apriorierror3}, we set $l=-1$ in \eqref{thetaineq-2}, 
and we recall the assumption of having a quasi-uniform  family of triangulations. 
Then $H^1$-stability of the $L_2$-projection $\cP_h$, that is
\begin{equation} \label{H1stability}
  \| \cP_h v\|_1 \leq C \|v\|_1, \quad v \in H^1,
\end{equation}
holds true. 
Hence, recalling $e(T)=\theta(T)+\omega(T)$, we have
\begin{equation*}
  \begin{split}
    \|e_1(T)\|_1
    &\le C \Big\{\|(\cR_h-I)u^0\|_1 + \|(\cR_h-I)u_1(T)\|_1\\
    &\quad +\int_0^T\!\Big(
       \|(I_k-I)u_2\|_1+\|(I_k-I)u_1\|_2
       +\|(\cR_h-I)\dot u_1\|_1\Big)\,dt \Big\}\\
    \|e_2(T)\|
    &\le C \Big\{\|(\cR_h-I)u^0\|_1 + \|(\cR_h-I)u_2(T)\|\\
    &\quad +\int_0^T\!\Big(
       \|(I_k-I)u_2\|_1+\|(I_k-I)u_1\|_2
       +\|(\cR_h-I)\dot u_1\|_1\Big)\,dt \Big\}.
  \end{split}
\end{equation*}
This completes the proof of the error estimates 
\eqref{apriorierror2}--\eqref{apriorierror3} by \eqref{errorIk} and \eqref{errorRh}.
Now the proof is complete. 
\end{proof}

We note that the assumption of quasi-uniformity for validity of 
\eqref{H1stability}, that is used for error estimates 
\eqref{apriorierror2}--\eqref{apriorierror3}, can be relaxed, see 
\cite{Carstensen2004}, though it is not an considerable restriction 
in a priori error analysis. 

\section{Numerical example}
Here we verify the order of convergence of the cG(1)cG(1) 
method by a simple 
example for a one dimensional problem with smooth 
convolution kernel.  
Another example for two dimensional case with similar results, 
with fractional order kernels of Mittag-Leffler type,  
can be found in \cite{StigFardin2010}. 

We consider a decaying exponential kernel with 
$\|\cK\|_{L_1(\IR^+)}=\kappa=0.5$, the initial 
data $u^0=u^1=0$, and load term $f=0$. We set homogeneous Dirichlet 
boundary condition at $x=0$ and a constant Neumann boundary 
condition at the end point $x=1$, toward negative $y$ axis. 
Figure 1 shows that the method preserves the behaviour of the 
model problem. 

In Figure 2, we have verified numerically the spatial rate of convergence 
$O(h^2)$ for $L_2$-norm of the displacement.  
In the lack of an explicit solution we compare with a numerical 
solution with fine mesh sizes $h,k$. 
Here $h_{min}= 0.0078$ and $k_{min}=0.017$. 
The result for temporal order of convergence, $O(k^2)$, is similar. 
\begin{figure}[htbp]
   \begin{center}
       \includegraphics[width=12cm,height=6cm]{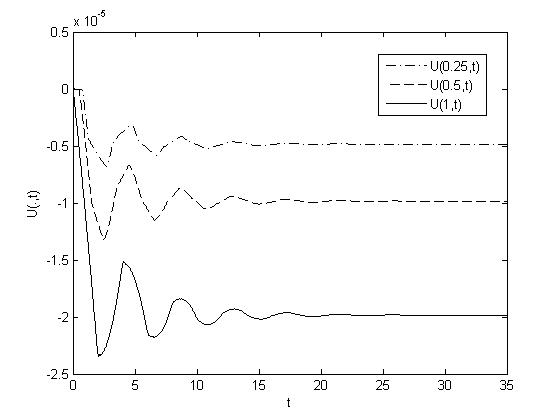}
       \caption{Damping of the oscillating material: at points $x=0.25, 0.5, 1$.}
   \end{center}
\end{figure}
\begin{figure}[htbp]
   \begin{center}
       \includegraphics[width=12cm,height=5cm]{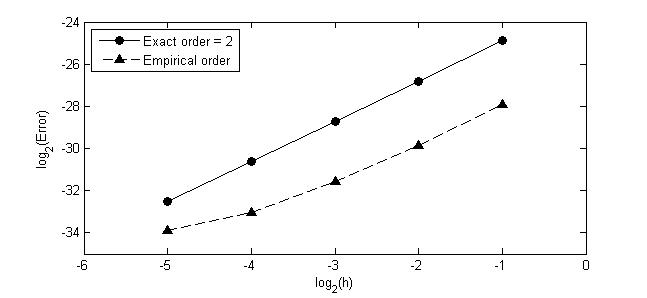}
       \caption{Order of convergence for spatial discretization}
   \end{center}
\end{figure}

\bibliographystyle{amsplain}
\bibliography{cGApriori2013March8}

\end{document}